\documentclass{amsart}
\usepackage{chngcntr}
\usepackage{apptools}
\usepackage{color}
\usepackage{amsthm,amssymb,verbatim}
\usepackage{graphicx}
\usepackage{enumerate}
\usepackage{chngcntr}
\usepackage{apptools}
\usepackage{color}
\usepackage{amsthm,amssymb,verbatim}
\usepackage{graphicx}
\usepackage{enumerate}

\newcommand{\Ff}{\mathcal F}

 \newcommand{\RR}{\mathbf{R}}  
 
 \newcommand{\BB}{\mathbf{B}}  

    \newcommand{\dist}{\operatorname{dist}}

 \newcommand{\eps}{\epsilon}
 \newcommand{\Tan}{\operatorname{Tan}}

\newcommand{\ee}{\mathbf e}

\newcommand{\reach}{\operatorname{reach}}
\newcommand{\spt}{\operatorname{spt}}
\newcommand{\Hh}{\mathcal{H}}
\usepackage{amsthm}

\input epsf
\def\begfig {
\begin{figure}
\small }
\def\endfig {
\normalsize
\end{figure}
}

    \newtheorem{theorem}    {Theorem}   
    \newtheorem{lemma}      [theorem]       {Lemma}
    \newtheorem{corollary}  [theorem]     {Corollary}

    \newtheorem*{theorem*}{Theorem}
    \theoremstyle{definition}
    \newtheorem{definition}  [theorem] {Definition}
     
    \theoremstyle{definition}
    \newtheorem{remark}   [theorem]       {Remark}
    
\renewcommand{\thesubsection}{\thetheorem}

\usepackage{enumerate}
\usepackage{hyperref}
\usepackage[alphabetic, msc-links, backrefs]{amsrefs}
\begin{document}

\renewcommand{\thesubsection}{\thetheorem}
\title{On the fundamental regularity theorem for mass-minimizing flat chains}

\author{Brian White}
\address{Brian White\newline
Department of Mathematics \newline
 Stanford University \newline 
  Stanford, CA 94305, USA\newline
{\sl E-mail address:} {\bf bcwhite@stanford.edu}
}

\begin{abstract}
In the theory of flat chains with coefficients in a normed abelian group,
we give a simple necessary and sufficient condition on a group element $g$
in order for the following fundamental regularity principle to hold:
if a mass-minimizing chain is, in a ball disjoint from the boundary, sufficiently weakly
close to a multiplicity $g$ disk, then, in a smaller ball, it is a  
$C^{1,\alpha}$ perturbation with multiplicity $g$
of that disk.
\end{abstract}

\date{7 August, 2024}

\newcommand{\cone}{\operatorname{cone}}
\newcommand{\diag}{\operatorname{diag}}

\newcommand{\Aac}{\mathcal{A}_{\rm con}}

\newcommand{\Aar}{\mathcal{A}_{\rm rect}}

\newcommand{\pdf}[2]{\frac{\partial #1}{\partial #2}}
\newcommand{\pdt}[1]{\frac{\partial #1}{\partial t}}

\newcommand{\entropy}{\operatorname{entropy}}
\newcommand{\mcd}{\operatorname{mcd}}
\newcommand{\mdr}{\operatorname{mdr}}
\newcommand{\tM}{\tilde M}
\newcommand{\tF}{\tilde F}
\newcommand{\tmu}{\tilde \mu}

\newcommand{\Gin}{\Gamma_\textnormal{in}}

\newcommand{\Tfat}{\operatorname{T_\textnormal{fat}}}
\newcommand{\tTfat}{\operatorname{\tilde T_\textnormal{fat}}}

\maketitle

 
\section{Introduction}
 
The fundamental regularity theorem for mass-minimizing integral flat chains can be stated
in various ways, one of which is the following:

\begin{theorem}\label{fundamental-1}
Let $0<\alpha<1$.
Suppose $M$ is a a mass-minimizing integral $m$-chain in $\RR^d$ such that
$M$ has no boundary in the open unit ball $B(0,1)$.  Suppose also that $M\llcorner B(0,1)$
is weakly sufficiently close to a multiplicity-$1$ disk through the origin.
Then $M\llcorner B(0,1/2)$ is an $m$-manifold of multiplicity $1$.
Furthermore, the $m$-manifold is the graph of
a function $f$ over a domain in the given disk with $\|f\|_{C^{2,\alpha}}\le 1$.
\end{theorem}

More generally, the theorem is also true for ``almost minimizing'' chains.
In particular, it applies to chains that are homologically mass minimizing in smooth Riemannian
manifolds, or, more generally, in sets of positive reach.  See~\S\ref{almost-section}.

If $G$ is any normed abelian group, then the theory of integral flat 
chains generalizes to flat chains with coefficients in $G$.
In this paper, we address the question: for which coefficient groups $G$ and for which
multiplicities $g\in G$ does Theorem~\ref{fundamental-1} hold?

It is natural to require that $G$ be complete with respect to its norm $|\cdot|$, i.e., that
$G$ be a complete metric space with respect to the metric $d(x,y)=|x-y|$.
(Otherwise, replace $G$ by its metric space completion.)
It is also natural to require that
\begin{equation}\label{compact}
  \text{If $R<\infty$, then $\{x\in G: |x|\le R\}$ is compact.}
\end{equation}
Condition~\eqref{compact}
  is the necessary and sufficient condition for the fundamental compactness
theorem for flat chains to hold.  (This follows easily from the 
Deformation Theorem~\cite{fleming}*{Theorem~7.3} or~\cite{white-deformation}.)

For such normed abelian coefficient groups $G$, we prove

\begin{theorem}\label{main-intro-theorem}
Suppose $g\in G$.
Theorem~\ref{fundamental-1} holds for $m$-chains with coefficients in $G$
and with ``multiplicity $g$'' in place of ``multiplicity $1$''
if and only $g$ satisfies the following
strong triangle inequality:
\begin{equation}\label{strong}
   |g| < \inf \{ |a|+|b|:  a, b\in G\setminus\{0\}, \, a+b=g\}.
\end{equation}
\end{theorem}

We remark that (as is easily shown) the strong triangle inequality~\eqref{strong}
is equivalent to

\begin{enumerate}[\upshape (i)]
\item $\inf_{x\in G\setminus\{0\}}|x| > 0$, and
\item $|g|<|a|+ |b|$ for all $a,b\in G\setminus\{0\}$ for which $a+b=g$.
\end{enumerate}

Simple examples show that the strong triangle inequality is necessary for the regularity
theorem to hold; see \S\ref{examples-section}.
Most of the paper is devoted to showing that sufficiency of the strong triangle inequality
follows from the Allard Regularity Theorem~\cite{allard}.
See~\cite{short-allard} for a greatly simplified proof of Allard's theorem.

Theorem~\ref{fundamental-1} holds more generally for integral flat chains that minimize (or almost minimize) 
parametric elliptic functionals. 
See~\cite{schoen-simon}, ~\cite{bombieri}, \cite{federer-book}*{5.3.14},
or (for the original, slightly different theorem)~\cite{almgren}.
I conjecture that Theorem~\ref{main-intro-theorem} also holds for parametric elliptic functionals.
(If the group has elements of order $2$, then the parametric elliptic integrand needs to be even for
the functional to make sense.)
The proofs in this paper rely on monotonicity and on the Allard Regularity Theorem and hence
only work for mass.

In a different direction, De Pauw and Z\"ust~\cite{depauw-zust} have proved a regularity theorem 
for mass-minimizing or almost mass-minimizing flat chains (with coefficients in an abelian group)
in possibly infinite dimensional Hilbert spaces.  Specifically,  they prove that 
if $\{ |x|: x\in G\}$ is a discrete set,  then the regular set of the chain is a dense open subset of the support.
(In the finite dimensional setting, and with the notion of almost minimizing used in this paper,
their result follows immediately from the Allard Regularity Theorem.)

\section{Preliminaries}

Throughout the paper, $G$ is a normed abelian group, i.e.,  an abelian group with a translation invariant metric; the norm
of an element is its distance to $0$.
We will always assume that $G$
is complete and satisfies the compactness property~\eqref{compact}.
In this paper, ``$m$-chain'' will mean ``flat $m$-chain with coefficients in $G$''.
We do not require that the chains have compact support.  Thus, in the terminology
of Federer's book~\cite{federer-book}, they
would be called ``locally flat chains''.  
The appendix of~\cite{white-duke-flat}
 indicates how to extend the theory
of compactly supported chains to the general case.
See~\cite{white-immiscible} for a quick introduction to flat chains with coefficients in a normed
abelian group.  Fleming's original paper~\cite{fleming} is the standard reference.
See also~\cite{white-deformation} and \cite{white-rectifiability},
 or, for a different approach,~\cite{depauw-hardt}.

If $M$ is an $m$-chain in $\RR^d$, we let $|M|$ denote its mass.  
If $M$ has locally finite mass 
(which will be the case for all $m$-chains that arise in this paper) and if $S$
is a Borel subset of $\RR^d$, then $M$ has a well-defined portion in $S$, 
denoted $M\llcorner S$.  Furthermore, $M$ determines a Radon
  measure $\mu_M$
on $\RR^d$ such that
\[
  \mu_M(S) = |M\llcorner S|
\]
for every Borel set $S$.

Except in \S\ref{examples-section}, we will be working with groups that also have the property:
\begin{equation}\label{discrete}
  \inf_{x\in G, \, x\ne 0} |x| > 0.
\end{equation}
(If $G$ does not satisfy~\eqref{discrete}, then no element of $G$ satisfies the
strong triangle inequality~\eqref{strong}.)
For coefficient groups satisfying~\eqref{discrete}, every compact set is finite.  
In particular, by~\eqref{compact},
\[
  \text{If $R<\infty$ then $\{x\in G: |x|\le R\}$ is a finite set.}
\]

For such normed groups $G$, every $m$-chain $M$ of finite (or locally finite) mass 
is rectifiable~\cite{white-rectifiability}.  That is, $M$ can be written as
\[
   M = \sum_{i=1}^\infty g_i [S_i],
\]
where the $g_i\in G$, where the $S_i$ are disjoint Borel sets, 
and and where each $S_i$ is contained in a $C^1$, oriented $m$-manifold $\Sigma_i$. 
If $W$ is a Borel subset of the ambient space, 
then
\begin{align*}
M\llcorner W &= \sum_{i=1}^\infty g_i [S_i\cap W], \\
|M| &= \sum_{i=1}^\infty |g_i| \,\Hh^m(S_i), \\
   \mu_M(W) &= \sum_{i=1}^\infty |g_i| \, \Hh^m(S_i\cap W).
\end{align*}

We also define flat norms $\Ff(\cdot)$ and $\Ff(\cdot,K)$
and flat seminorms $\Ff_W(\cdot)$ as follows.

\begin{definition}\label{flat}
Let $M$ be a flat $m$-chain in $\RR^d$.
 We define
 $\Ff(M)$ to be the infimum 
of
\begin{equation}\label{flat-expression}
   |M-\partial Q| + |Q|.
\end{equation}
over all $(m+1)$-chains $Q$.

If $K$ is a closed subset of $\RR^d$, we define $\Ff(M;K)$
to be the infimum of~\eqref{flat-expression} among all $(m+1)$-chains $Q$ supported in $K$.

If $U$ is an open subset of $\RR^d$, we let
\[
  \Ff_U(M)
\]
be the infimum of 
\[
  |(M-\partial Q)\llcorner U| + |Q\llcorner U|
\]
over all finite-mass $(m+1)$-chains $Q$.
\end{definition}

The term $(M-\partial Q)\llcorner U$ needs explanation, since we have only defined $A\llcorner U$
when $A$ has locally finite mass.  
Suppose $A$ is any $m$-chain and $U$ is an open set.  If there is a $m$-chain $B$ supported in $U^c$
such that $A-B$ has locally finite mass, then we let $A\llcorner U= (A-B)\llcorner U$; otherwise 
we leave $A\llcorner U$ undefined.  If $A\llcorner U$ is undefined, we take $|A\llcorner U|$ to be infinite.

Note that if $U\subset W$, then
\[
  \Ff_U(M)\le \Ff_W(M).
\]

  If $\overline{U}$ is compact, then $\Ff_U(M)<\infty$.
It follows that if $M$ is compactly supported, then $\Ff(M)<\infty$.
We say that $M_n$ converges (weakly) to $M$ and write $M_n\to M$ provided
\[
  \Ff_U(M_n-M)\to 0
\]
for all bounded open sets $U$ of $\RR^d$.  If $M_i$ and $M$ are supported
in a compact set, then $M_n\to M$ if and only if $\Ff(M_n-M)\to 0$.

\section{$\lambda$-Minimizing Chains}\label{almost-section}

\begin{definition}
For $\lambda\in [0,\infty)$,
an $m$-chain $M$ of locally finite mass
in $\RR^d$ is called {\bf $\lambda$-minimizing} provided it has the following property.
If $K$ is a compact subset of $U$
and if $Q$ is an $(m+1)$-chain compactly supported in 
\[
   \{x: \dist(x, \spt M)< r\},
\]
then 
\begin{equation}\label{almost-bound}
   (1 - \lambda r) \, |M\llcorner K|  \le  |M\llcorner K + \partial Q|.
\end{equation}
\end{definition}

Thus ``$0$-minimizing'' is the same as ``mass-minimizing".

\begin{remark}\label{almost-remark}
If $M$ is $\lambda$-minimizing, then~\eqref{almost-bound}
also follows from the slightly weaker hypothesis
\[
 \spt Q  \subset \{x:\dist(x,\spt M)\le r\},
\]
for, in that case, 
\[
   (1-\lambda R)\, |M\llcorner K| \le |M\llcorner K + \partial Q|
\]
holds for every $R>r$ and therefore also for $R=r$. 
\end{remark}

If $C$ is a closed subset of $\RR^d$, if $x\in \RR^d$, and if there is a unique point $y\in C$
closest to $x$, we let $\pi_C(x)=y$; otherwise $\pi_C(x)$ is not defined (i.e., $x$ is not in the domain
of $\pi_C$.)  
We define $\reach(C)$ to be the smallest number $R$ such that 
\[
    \{x: \dist(x,C)<R\}
\]
is contained in the domain of $\pi_C$.  

\begin{theorem}
Suppose $C\subset \RR^d$ is a set with $R:=\reach(C)>0$.
Suppose that $M$ is homologically minimizing in $C$ i.e., that $M$ is supported in $C$,
and that if $K$ is a compact subset of $\spt M$, 
then
\[
   |M\llcorner K| \le |M\llcorner K + \partial Q|
\]
for every $(m+1)$-chain $Q$ supported in $C$.
Then $M$ is $(m/R)$-minimizing.
\end{theorem}

\begin{proof}
For $r<R$, the restriction $\pi$ of $\pi_C$ to $\{x: \dist(x,C)\le r\}$
is Lipschitz with Lipschitz constant $\le R/(R-r)$. 
 (See~\cite{federer-curvature}*{Theorem~4.8(8)}.)
Suppose that $Q$ is an $(m+1)$-chain supported in 
\[
   \{x: \dist(x,K)\le r\} \subset \{x: \dist(x,C)\le r\}.
\]
Then
\begin{align*}
|M\llcorner K|
&\le
|M\llcorner K + \partial (\pi_\#Q)|  \\
&=
| \pi_\#( M\llcorner K + \partial Q) |    \\
&\le
\left( \frac{R}{R-r}\right)^m |M\llcorner K + \partial Q|
\\
&=
\left( 1 - \frac{r}R\right)^{-m} |M\llcorner K + \partial Q|,
\end{align*}
so
\begin{align*}
|M\llcorner K + \partial Q|
&\ge
\left(1 - \frac{r}{R} \right)^m |M\llcorner K|  \\
&\ge
\left( 1  -  \frac{m}R r \right) |M\llcorner K|.
\end{align*}
\end{proof}

\begin{theorem}
Suppose $M$ is a rectifiable, $\lambda$-minimizing $m$-chain in $\RR^d$.
Then in $\RR^d\setminus\spt(\partial M)$, the varifold $V$ associated to $M$ has  mean curvature bounded by $\lambda$.
\end{theorem}

\begin{proof}
Let $X$ be smooth vectorfield supported in a compact subset $C$ of $\RR^d\setminus \spt(\partial M)$
with $|X(p)|\le 1$ for all $p$.
Let $K=C\cap \spt M$.
Let $\phi_t(p)= p + tX(p)$.  Let $\Sigma_t = \phi_\#(M\llcorner K)$.
Let $Q_t$ be the image of 
\[
    (M\llcorner K)\times [0,t] 
\]
under the map $(p,\tau)\mapsto \phi_\tau(p)$.

Then
\[
  (1 - \lambda t) |M\llcorner K| \le |\Sigma_t|
\]
Thus, differentiating, and using the first variation formula,
\[
     - \lambda \mu_V(K) \le \delta(V;X).
\]
Since the same holds for $-X$, we see that
\[
   |\delta(V;X)| \le \lambda \mu_V(K),
\]
which is equivalent to the assertion of the theorem.
\end{proof}

\begin{theorem}\label{convergence-theorem}
Suppose that $M_i$ and $M$ are $m$-chains of locally finite mass in $\RR^d$,
 that $M_i$ is $\lambda$-minimizing,
 and that $M_i\to M$.
Then $M$ is $\lambda$-minimizing and
\[
  \mu_{M_i} \to \mu_M.
\]
\end{theorem}

\begin{proof}
By passing to a subsequence, we may suppose that $\spt M_i$ converges
to a closed set $X$.  Thus
\[
  \dist(\cdot, \spt M_i) \to \dist(\cdot, X).
\]
In particular, if $V$ is any compact subset of $U$,
then
\[
  \max_{q\in V} \dist(q,\spt M_i) \to \max_{q\in V} \dist(q, X).
\]

Let $0< r < 1/\lambda$.
Let $Q$ be an $(m+1)$-chain compactly supported in 
\[
  \{x\in U:  \dist(x, X) < r\}.
\]
Thus
\[
   \max_{\spt Q} \dist(\cdot, X) < r.
\]
Hence
\[
 \max_{\spt Q}\dist(\cdot, M_i) < r
\]
for all sufficiently large $i$.

Now let $K$ be a compact subset of $X$, let $f:\RR^d\to \RR$ be the distance function to $K$,
and let
\begin{align*}
K[s] &= \{f\le s\} = \{x:\dist(x,K)\le r\}, \\
M[s] &= M\llcorner K[s], \\
M_i[s] &= M_i\llcorner K[s].
\end{align*}
By passing to a subsequence, we can 
assume that for almost every $s\in (0,r)$,
\[
  \Ff( M_i[s] - M[s]; K[s]) \to 0.
\]
(See Theorem~\ref{flamingo-theorem} in the appendix.)
For such an $s$, there exist $(m+1)$-chains $Q_i$ supported in $K[s]$
such that
\[
  |M_i[s] - M[s] + \partial Q_i| + |Q_i| \to 0.
\]
Now
\[
   \max_{K[s]}\dist(\cdot,X) \le s < r.
\]
Thus
\[
  \max_{K[s]} \dist(\cdot, \spt M_i) < r
\]
for all sufficiently large $i$.  In particular,
\[
  \max_{\spt Q_i} \dist(\cdot, \spt M_i) < r
\]
for all sufficiently large $i$.

Now
\begin{align*}
(1-\lambda r) |M_i[s]|
&\le
|M_i[s]  - \partial (Q+Q_i) |
\\
&\le
|M[s] - \partial Q| + |M_i[s] - M[s] - \partial Q_i|
\end{align*}
Thus
\begin{equation}\label{george}
(1-\lambda r) \limsup_i |M_i[s]| 
\le 
|M[s] - \partial Q|.
\end{equation}

In particular (letting $Q=0$), we see that every compact set $K$ in $X$
is contained in the interior of a set $K[s]$ for which $\limsup_i |M\llcorner K[s]| <\infty$.
Thus, after passing to a subsequence, the Radon measures $\mu_i=\mu_{M_i}$
converge weakly to a Radon measure $\mu$ on $U$.

By~\eqref{george} and lower semicontinuity of mass with respect to flat convergence, 
\begin{equation}\label{pongo}
\begin{aligned}
(1-\lambda r)|M[s]|
&\le
(1-\lambda r) \liminf_i|M_i[s]|
\\
&\le 
(1-\lambda r) \limsup_i |M_i[s]| 
\\
&\le 
|M[s] - \partial Q|.
\end{aligned}
\end{equation}
Letting $s\to 0$ gives
\[
(1-\lambda r)|M[0]| \le |M[0] - \partial Q|.
\]
Hence $M$ is $\lambda$-minimizing.

Now consider~\eqref{pongo} in the case $Q=0$.
For almost every $s$, $\mu_M(K[s])\to \mu(K[s])$.
For such $s$, we can rewrite~\eqref{pongo} as
\[
(1-\lambda r)\mu_M(K[s])
\le 
(1-\lambda r) \mu(K[s])
\le
\mu_M(K[s]).
\]
Letting $s\to 0$ and then $r\to 0$ gives
\[
   \mu_M(K)\le \mu(K) \le \mu_M(K).
\]
\end{proof}

 \section{Main Theorem}
 
Throughout this section, we assume that $g$ satisfies the strong triangle inequality,
and we prove the fundamental regularity theorem under that assumption.
Note that the strong triangle inequality for $g$ is 
equivalent to the existence of an $\eps>0$ such that
\begin{equation}\label{strong-2}
   |a|+|b|\ge |g|+ \eps \quad\text{if $a+b=g$ and $a,b\in G\setminus\{0\}$}.
\end{equation}
This implies that
\begin{equation}\label{discrete-2}
 \inf_{x\in G, \, x\ne 0} |x| > 0.
\end{equation}
(Indeed, if $x\ne 0$, then $|x|\ge \min\{\eps/2,\,|g|\}$.)

The condition~\eqref{discrete-2} implies that every 
flat chain of locally finite mass is rectifiable~\cite{white-rectifiability}.
 
\begin{lemma}\label{projection-lemma}
Suppose $M$ is a rectifiable, compactly supported $m$-chain such that
\[
  \pi_\#M = g[\Omega],
\]
where $\Omega$ is a Borel subset of an oriented $m$-plane $P$
and where $\pi:M\to P$ is  the orthogonal projection.
Suppose also that $M$ has no points of multiplicity $g$.
 Then
\[
  |M| \ge (|g|+\eps)\Hh^m(\Omega).
\]
\end{lemma}

Here, ``point of multiplicity $g$'' means ``point that has a plane of  multiplicity $g$ 
as a tangent cone".

\begin{proof}
Since $M$ is rectifiable, it can be written
as
\[
   M = \sum_i g_i[S_i],
\]
where the $S_i$ are disjoint Borel sets and each $S_i$ is contained in an oriented $C^1$
  $m$-manifold $\Sigma_i$.
By subdividing the $S_i$ and $\Sigma_i$ and reorienting (if necessary), we can assume that
if $p\in S_i$ and if 
\[
    \pi\vert \Tan(\Sigma_i,p): \Tan(\Sigma_i,p) \to P
\]
is an isomorphism, then it preserves orientation.   
Let $S=\cup S_i$ and define
\begin{align*}
   &\gamma: S\to G, \\
   &\gamma(p)=g_i \quad\text{for $p\in S_i$}.
\end{align*}

Since $M$ has no points of multiplicity $g$, we can assume that $\gamma(p)$ is never
equal to $g$ or $-g$.

Since $\pi_\#M=g[\Omega]$, we see that 
\[
  \sum_{p \in S\cap \pi^{-1}x} \gamma(p) = g
\]
for almost every $x\in \Omega$, and therefore that
\[
    \sum_{p\in \pi^{-1}x} |\gamma(p)| \ge |g|+\eps
\]
since $\gamma(p)\ne  g$.

Thus
\begin{align*}
|M|
&=
\int_{S} |\gamma(p)|\, d\Hh^mp
\\
&\ge
\int_{x\in \Omega} \sum_{p\in S\cap \pi^{-1}x} |\gamma(p)| \, dx
\\
&\ge
\int_\Omega (|g|+\eps) \, d\Hh^m
\\
&\ge
(|g|+\eps)  \Hh^m(\Omega).
\end{align*}
\end{proof}

If $M$ is an $m$-chain of locally finite mass, and if $B(p,r)$ is a ball, we let
\[
  \Phi(M,p,r)
\]
be the minimum of 
\[
  \Ff_{B(0,1)}\left( \frac1r(M-p), g[P] \right)
\]
among all oriented  $m$-planes $P$ with $0\in P$.

We also let
\begin{align*}
&\Theta(M,p,r) = \frac{|M\llcorner B(p,r)|}{\omega_mr^m}, \\
&\Theta(M,p) = \lim_{r\to 0}\Theta(M,p,r) \quad\text{(if the limit exists), and} \\
&\Theta(M) = \sup_{p\in \RR^d, \,r>0}\Theta(M,p,r).
\end{align*}

\begin{lemma}\label{bernstein-lemma}
There is a $\delta>0$ with the following property.
Suppose $M$ is a mass-minimizing $m$-chain in $\RR^d$ such that 
\begin{enumerate}
\item\label{bernie-1} $\partial M=0$.
\item\label{bernie-2} $\Theta(M)\le |g|$.
\item $\Phi(M,0,r)\le \delta$ for all $r>1$.
\end{enumerate}
Then $M=g[P]$ for some oriented $m$-plane $P$.
\end{lemma}

\begin{proof}
Suppose that $M_i$ is a sequence of $m$-chains
satisfying~\eqref{bernie-1} and~\eqref{bernie-2}, and such that
\[
  \delta_i:= \sup_{r>1} F(M_i,0,r) \to 0.
\]
It suffices to prove that $M_i$ is a $m$-plane of multiplicity $g$ for all
sufficiently large $i$.
By rotating the $M_i$, we can assume that
\begin{equation}\label{to-P}
  \Ff_{B(0,1)} (M_i - g[P]) \to 0
\end{equation}
where $P=\RR^m\times\RR^{N-m}$.

It follows that the $M_i$ converge weakly to $g[P]$.

(If this is not clear, note that if $R>1$, then, after passing to a subsequence,
there is a $P'$ such that $\Ff_{B(0,R)}(M_i - g[P']) \to 0$.
By~\eqref{to-P}, $P'=P$, and thus it was not necessary to pass to a subsequence.)

Let
\[  
  K[r] = \overline{B^m(0,r)\times B^{d-m}(0,r)}.
\]
That is, $K[r]= \{f\le r\}$, where 
\begin{align*}
&f:\RR^d\to \RR, \\
&f(x,y)=\max\{|x|, |y|\} \quad (x\in \RR^m, \, y\in \RR^{d-m}).
\end{align*}

After passing to a subsequence, we
can assume (by Theorem~\ref{flamingo-theorem}) that
\[
   M_i\llcorner \Omega(r) \to g[D_r]
\]
for almost every $r$, 
where $D_r=P\cap B(0,r)$.  
Fix such an $r$, and let $M_i'=M_i\llcorner K[r]$ and $D=D_r$.

Since $\spt(M_i)\to P$, if follows that
\[
   \spt(M_i') \subset  \overline{B(0,r)\times B(0,\eta_i)}
\]
where $\eta_i\to 0$.
For all sufficiently large $i$, $\eta_i<r$, and thus
\[
  \spt(\partial M_i') \subset \partial B^m(0,r)\times \overline{B(0,\eta_i)}.
\]
Consequently,
\[
  \pi_\# M_i = g_i[D]
\]
for some $g_i$.  Since $M_i\to g[D]$, we see
that $|g_i-g|\to 0$ and therefore (since $\inf_{x\in G,\,x\ne 0}|x|>0$)
  that $g_i=g$ for all sufficiently large $i$.

The $\lambda$-minimizing property implies (see Theorem~\ref{convergence-theorem}) that
\[
   |M_i'| \to |g[D]| =  |g|\Hh^m(D).
\]
Hence, by Lemma~\ref{projection-lemma}, for large $i$, $M_i'$ must have a point $p_i$
at which the tangent cone is a multiplicity $m$-plane.
Since $\Theta(M_i')\le |g|$, it follows from monotonicity
that $M_i'$ is a cone with vertex $p_i'$,
and therefore a multiplicity $g$ plane.
\end{proof}

\begin{corollary}\label{bernstein-corollary}
If $M$ is a $\lambda$-minimizing $m$-chain in $\RR^d$, if $p$ is a point
in $\spt(M)\setminus \spt(\partial M)$, and if
\[
   \sigma:=\limsup_{r\to 0} F(M,p,r) < \delta,
\]
then every tangent cone to $M$ at $p$ is a multiplicity-$g$ plane.
\end{corollary}

\begin{proof}
Let $M'$ be a tangent cone to $M$ at $p$.  Then $F(M',0,r)\le \delta$
for all $r$, so $M'$ is a multiplicity-$g$ plane by Lemma~\ref{bernstein-lemma}.
\end{proof}

\begin{theorem}\label{fundamental-2}
Suppose that $G$ is a complete normed abelian group satisfying the compactness
property~\eqref{compact}, and  that $g\in G$ satisfies
  the strong triangle inequality~\eqref{strong-2}.
Suppose that $U$ is an open subset of $\RR^d$
and that $M_i$ is a sequence of $\lambda$-minimizing $m$-chains in $\RR^d$ with 
$\spt(\partial M_i)\subset U^c$.
  Suppose also that  $M_i\llcorner U$ converges
to $g[\Sigma]$, where $\Sigma$ is a properly $C^1$-embedded oriented $m$-manifold in 
  $U$.
Then there exists an exhaustion $W_1\subset W_2\subset \dots$ of $U$
by open subsets and $C^{1,\alpha}$ properly embedded, oriented $m$-manifolds $\Sigma_i$ in $W_i$
such that 
\[
     M_i\llcorner W_i =  g[\Sigma_i]
\]
and such that the $\Sigma_i$ converge in $C^{1,\alpha}$ with multiplicity $1$ to $\Sigma$.
\end{theorem}

\begin{proof}
By the Allard Regularity Theorem (\cite{allard} or~\cite{short-allard}), it suffices to show that if $W\subset\subset U$,
then for all sufficiently large $i$,
\[
 \Theta(M,p)\ge |g| \quad\text{for every $p\in W\cap\spt M_i$}.
 \]
Suppose not.  Then (after passing to a subsequence), there exist
$p_i\in W\cap \spt M_i$ for which
\[
  \Theta(M_i,p_i) < |g|.
\]
By passing to a further subsequence, we can assume that the $p_i$
converge to a point $p\in\overline{W}$.

Choose $R_i$ converging to $0$ sufficiently slowly
that 
\[
   \frac1{R_i}(M_i - p_i)  \to g[\Tan(\Sigma,p)].
\]
Hence if 
\[
  0 < \liminf \frac{\rho_i}R_i \le \limsup \frac{\rho_i}R_i < \infty,
\]
then
\[
   \frac1{\rho_i}(M_i - p_i) \to g[\Tan(\Sigma(p)],
\]
and therefore
\begin{equation}\label{rho}
  \Phi(M_i,p_i,\rho_i) \to 0.
\end{equation}

Let $\delta$ be as in Lemma~\ref{bernstein-lemma}.
Let 
\[
  M_i' := \frac1{r_i} (M_i- p_i),
\]
where $r_i$ is the supremum of $r\in (0,R_i]$ such that
\[
   F(M_i,p_i,r) \ge \delta.
\]
Note that, by Corollary~\ref{bernstein-corollary}, there exists such an $r$, so $r_i>0$.
By~\eqref{rho}, 
\[
  \frac{r_i}R_i \to 0.
\]

Let $M'$ be a subsequential limit of $M_i'$.
Then
\[
    F(M',0,r) \le \delta
\]
for all $r\ge 1$, so $M'$ is a multiplicity $g$
plane through the origin (by Lemma~\ref{bernstein-lemma}),
and therefore $F(M',0,r)=0$ for all $r>0$.
But $F(M',0,1)\ge \delta$ by choice of $r_i$,
a contradiction.
\end{proof}

\section{Examples}\label{examples-section}

Here we show that the hypothesis
\begin{equation}\label{hypothesis-end}
  |g| <  \inf_{a,b\in G\setminus\{0\}, \,a+b=g} (|a|+|b|)
\end{equation}
is necessary in Theorem~\ref{fundamental-2}.

Recall that~\eqref{hypothesis-end} is equivalent to
\begin{enumerate}
\item $\inf_{x\in G\setminus \{0\}}|x|>0$, and
\item $|g|< |a|+|b|$ for all $a,b\in G\setminus \{0\}$ such that $a+b=g$.
\end{enumerate}

Let $U=B^{m}(0,1)\times\RR$
and let $\Sigma$ be the unit $m$-ball $\BB^{m}(0,1)\times\{0\}$ with the standard orientation.

If $\inf_{x\in G\setminus\{0\}}|x|=0$,
choose $a_n\in G\setminus\{0\}$ with $a_n\to 0$.
Then $(g+a_n)[\Sigma]$ is area-minimizing and converges as $n\to\infty$ to $g[\Sigma]$,
violating the conclusion of Theorem~\ref{fundamental-2}.

Now suppose that there exist $a,b\in G\setminus\{0\}$ such that 
\begin{align*}
  g &= a+b, \\
  |g| &=|a| + |b|.
\end{align*}
Then
\[
   M_n:= a[\Sigma] + b[\Sigma+ (1/n)\ee_{m+1}]
\]
is mass-minimizing and converges to $g[P]$ as $n\to\infty$,
violating the conclusion of Theorem~\ref{fundamental-2}.

Here, $M_n$ is mass-minimizing according to the following lemma:

\begin{lemma}
Suppose that $a,b\in G$ and that $|a+b|=|a|+|b|$.
Suppose also that $\Omega$ is an open subset of an oriented $m$-plane $P$, and let
\[
   M = a[\Omega] + b[\Omega+u],
\]
where $u$ is a vector perpendicular to $P$.
Then $M$ is mass-minimizing.
\end{lemma}

\begin{proof}
Let $M'$ be a compactly supported chain with $\partial M'=\partial M$.
Let $\pi$ denote orthogonal projection onto $P$.
Then
\[
    \pi_\# (M') = (a+b)[\Omega],
\]
so
\begin{align*}
|M'|
&\ge
|\pi_\# M'|  
\\
&=  |a+b| \Hh^m(\Omega) 
\\
&= (|a| + |b|)\Hh^m(\Omega)
\\
&= (|a|)\Hh^m(\Omega) + |b|\Hh^m(\Omega+ u)
\\
&=  |M|.
\end{align*}
\end{proof}

\section{appendix}

\begin{lemma}
Suppose that $M$ is a finite mass $m$-chain in $\RR^d$,
that $f:\RR^d\to \RR$ is a Lipschitz function with Lipschitz constant $\le 1$.
For $s\in \RR$, let 
\begin{align*}
K[s] &= \{f\le s\}, \\
M[s] &= M\llcorner K[s].
\end{align*}
If $W\subset \RR^d$ is an open set that contains $\{f \le b\}$, then
\begin{equation}\label{coarea-inequality}
 \int_{s=a}^b \Ff(M[s]; K[s])\,ds \le (1+(b-a)) \Ff(M).
\end{equation}
\end{lemma}

\begin{proof}
If $A$ is a chain with $|A|<\infty$ and if $s\le b$, we let
\[
   A[s]:=  A \llcorner \{f\le s\} = A\llcorner K[s].
\]

Let $Q$ be an $(m+1)$-chain such that
\[
   |M- \partial Q| + |Q| < \infty.
\]
Since $|M|<\infty$, it follows that $|\partial Q|<\infty$.
Note that
\begin{equation}\label{long}
\begin{aligned}
\Ff(M[s]; K[s])
&\le
|M[s] -\partial (Q[s])| + |Q[s]|  \\
&=
|M[s] - (\partial Q)[s] + |Q(s)| + |(\partial Q)[s] -   \partial (Q[s])| \\
&=
|(M - \partial Q)[s]| + |Q(s)| + |(\partial Q)[s] -   \partial (Q[s])| \\
&=
|(M-\partial Q| + |Q| +  |(\partial Q)[s] -   \partial (Q[s])|  \\
\end{aligned}
\end{equation}
Also
\begin{align*}
\int_a^b | (\partial Q)[s] - \partial (Q[s])|\,ds
&\le
|Q|
\end{align*}
by~\cite{fleming}*{Theorem~5.7}.
Thus integrating~\eqref{long} gives
\begin{align*}
 \int_a^b \Ff(M[s], K[s]) \,ds 
&\le 
(1+b-a) (|M-\partial Q| + |Q|).
\end{align*}
Taking the infimum over $(m+1)$-chains $Q$ gives~\eqref{coarea-inequality}.
\end{proof}

\begin{theorem}\label{flamingo-theorem}
Suppose that $f:\RR^d\to\RR$ is Lipschitz with Lipschitz constant $\le 1$
and that $K[s]:=\{f\le s\}$ is compact for every $s<\infty$.
If $T_i$ are $m$-chains of locally finite mass such that $T_i\to 0$,
then (after passing to a subsequence)
\begin{equation}\label{Ti-to-zero}
   \Ff(T_i \llcorner K[s]; K[s]) \to 0
\end{equation}
for almost every $s\ge 0$.
\end{theorem}

\begin{proof}
For almost every $R$,
\[
    T_i\llcorner B(0,R) \to 0.
\]
(Indeed, this holds for any $R$ such that $\mu_{T}\partial B(0,R) = 0$.)
Thus, after passing to a subsequence, there exist $R_n\to \infty$ such that
\[
     \Ff(T_n\llcorner B(0,R_n)) < \frac1{2^n}
\]
and therefore
\[
  \sum_n \Ff(T_n\llcorner B(0,R_n)) < \infty.
\]
Let $T_n'=T_n\llcorner B(0,R_n)$.

Fix an $R<\infty$. 
For all sufficiently large $n$, say $n\ge N$, $K[R]$ is contained in $B(0,R_n)$.
Thus if $r\le R$ and $n\ge N$, then
\[
   T_n\llcorner K[r] = T_n'\llcorner K[r],
\]
so
\begin{align*}
\int_0^R \sum_{n\ge N} \Ff(T_n\llcorner K[r];K[r])\,dr
&=
\sum_{n\ge N} \int_0^R \int  \Ff(T_n' \llcorner K[s]; K[s])\,ds
\\
&\le
(1 + R) \sum_{n\ge N} \Ff(T_n')
\\
&<\infty.
\end{align*}
Hence, for almost all $s\in (0,R)$, we have
\[
   \sum_i \Ff(T_i[s]; K[s])<\infty,
\]
 and therefore~\eqref{Ti-to-zero} holds.
Since $R$ is arbitrary, \eqref{Ti-to-zero} holds for almost 
every $s\in [0,\infty)$.
\end{proof}


\begin{bibdiv}
\begin{biblist}

\bib{allard}{article}{
   author={Allard, William K.},
   title={On the first variation of a varifold},
   journal={Ann. of Math. (2)},
   volume={95},
   date={1972},
   pages={417--491},
   issn={0003-486X},
   review={\MR{0307015}},
   doi={10.2307/1970868},
}

\bib{almgren}{article}{
   author={Almgren, F. J., Jr.},
   title={Existence and regularity almost everywhere of solutions to
   elliptic variational problems among surfaces of varying topological type
   and singularity structure},
   journal={Ann. of Math. (2)},
   volume={87},
   date={1968},
   pages={321--391},
   issn={0003-486X},
   review={\MR{0225243}},
   doi={10.2307/1970587},
}

\bib{bombieri}{article}{
   author={Bombieri, Enrico},
   title={Regularity theory for almost minimal currents},
   journal={Arch. Rational Mech. Anal.},
   volume={78},
   date={1982},
   number={2},
   pages={99--130},
   issn={0003-9527},
   review={\MR{0648941}},
   doi={10.1007/BF00250836},
}

\bib{depauw-hardt}{article}{
   author={De Pauw, Thierry},
   author={Hardt, Robert},
   title={Rectifiable and flat $G$ chains in a metric space},
   journal={Amer. J. Math.},
   volume={134},
   date={2012},
   number={1},
   pages={1--69},
   issn={0002-9327},
   review={\MR{2876138}},
   doi={10.1353/ajm.2012.0004},
}

\bib{depauw-zust}{article}{
   author={De Pauw, Thierry},
   author={Z\"ust, Roger},
   title={Partial regularity of almost minimizing rectifiable $G$ chains in
   Hilbert space},
   journal={Amer. J. Math.},
   volume={141},
   date={2019},
   number={6},
   pages={1591--1705},
   issn={0002-9327},
   review={\MR{4030524}},
   doi={10.1353/ajm.2019.0044},
}

\bib{short-allard}{article}{
   author={De Philippis, Guido},
   author={Gasparetto, Carlo},
   author={Schulze, Felix},
   title={A short proof of Allard's and Brakke's regularity theorems},
   journal={Int. Math. Res. Not. IMRN},
   date={2024},
   number={9},
   pages={7594--7613},
   issn={1073-7928},
   review={\MR{4742836}},
   doi={10.1093/imrn/rnad281},
}

\bib{federer-book}{book}{
   author={Federer, Herbert},
   title={Geometric measure theory},
   series={Die Grundlehren der mathematischen Wissenschaften},
   volume={Band 153},
   publisher={Springer-Verlag New York, Inc., New York},
   date={1969},
   pages={xiv+676},
   review={\MR{0257325}},
}

\bib{federer-curvature}{article}{
   author={Federer, Herbert},
   title={Curvature measures},
   journal={Trans. Amer. Math. Soc.},
   volume={93},
   date={1959},
   pages={418--491},
   issn={0002-9947},
   review={\MR{0110078}},
   doi={10.2307/1993504},
}

\bib{fleming}{article}{
   author={Fleming, Wendell H.},
   title={Flat chains over a finite coefficient group},
   journal={Trans. Amer. Math. Soc.},
   volume={121},
   date={1966},
   pages={160--186},
   issn={0002-9947},
   review={\MR{0185084}},
   doi={10.2307/1994337},
}

\bib{schoen-simon}{article}{
   author={Schoen, Richard},
   author={Simon, Leon},
   title={A new proof of the regularity theorem for rectifiable currents
   which minimize parametric elliptic functionals},
   journal={Indiana Univ. Math. J.},
   volume={31},
   date={1982},
   number={3},
   pages={415--434},
   issn={0022-2518},
   review={\MR{0652826}},
   doi={10.1512/iumj.1982.31.31035},
}	

\bib{white-immiscible}{article}{
   author={White, Brian},
   title={Existence of least-energy configurations of immiscible fluids},
   journal={J. Geom. Anal.},
   volume={6},
   date={1996},
   number={1},
   pages={151--161},
   issn={1050-6926},
   review={\MR{1402391}},
   doi={10.1007/BF02921571},
}

\bib{white-deformation}{article}{
   author={White, Brian},
   title={The deformation theorem for flat chains},
   journal={Acta Math.},
   volume={183},
   date={1999},
   number={2},
   pages={255--271},
   issn={0001-5962},
   review={\MR{1738045}},
   doi={10.1007/BF02392829},
}

\bib{white-rectifiability}{article}{
   author={White, Brian},
   title={Rectifiability of flat chains},
   journal={Ann. of Math. (2)},
   volume={150},
   date={1999},
   number={1},
   pages={165--184},
   issn={0003-486X},
   review={\MR{1715323}},
   doi={10.2307/121100},
}

\bib{white-duke-flat}{article}{
   author={White, Brian},
   title={Currents and flat chains associated to varifolds, with an
   application to mean curvature flow},
   journal={Duke Math. J.},
   volume={148},
   date={2009},
   number={1},
   pages={41--62},
   issn={0012-7094},
   review={\MR{2515099}},
   doi={10.1215/00127094-2009-019},
}



\end{biblist}

\end{bibdiv}
\end{document}